\documentclass[12pt,a4paper,oneside]{amsart}
\usepackage[a4paper]{geometry}
\usepackage{txfonts}
\DeclareMathAlphabet{\mathcal}{OMS}{cmsy}{m}{n} 

\usepackage{amssymb,enumerate,mathrsfs,perpage}

\MakePerPage[2]{footnote}

\DeclareMathOperator{\ad}{\mathsf{ad}}

\DeclareMathOperator{\cd}{cd}
\DeclareMathOperator{\dcobound}{d}
\DeclareMathOperator{\Der}{Der}
\DeclareMathOperator{\E}{E}
\DeclareMathOperator{\HH}{HH}
\DeclareMathOperator{\Homol}{H}
\DeclareMathOperator{\Hom}{Hom}

\newtheorem{theorem}{Theorem}[section]
\newtheorem{proposition}[theorem]{Proposition}
\newtheorem{lemma}[theorem]{Lemma}
\newtheorem{corollary}[theorem]{Corollary}
\newtheorem{conjecture}[theorem]{Conjecture}

\numberwithin{equation}{section}

\begin{document}

\title{On Lie $p$-algebras of cohomological dimension one}
\author{Pasha Zusmanovich}
\address{
Department of Mathematics, University of Ostrava, Ostrava, Czech Republic
}
\email{pasha.zusmanovich@osu.cz}
\date{First written January 3, 2016; last minor revision July 8, 2019}
\thanks{Indag. Math. \textbf{30} (2019), no.2, 288--299; arXiv:1601.00352}

\begin{abstract}
We prove that a Lie $p$-algebra of cohomological dimension one is 
one-dimensional, and discuss related questions.
\end{abstract}

\maketitle

\section{Introduction}

The \emph{cohomological dimension} of a Lie algebra $L$ over a field $K$, 
denoted by $\cd(L)$, is defined as the right projective dimension of the trivial
$L$-module $K$, i.e., the minimal possible length of a finite projective 
resolution 
\begin{equation}\label{eq-resol}
\dots \to P_2 \to P_1 \to P_0 \to K
\end{equation}
consisting of right projective modules $P_i$ over the universal enveloping 
algebra $U(L)$, or infinity if no such finite resolution exists. (Of course,
right $U(L)$-modules here can be replaced by left ones). Since for every 
projective resolution (\ref{eq-resol}) and every $L$-module $M$, the cohomology of the 
induced complex
$$
0 \rightarrow M = \Hom_{U(L)}(K,M) \rightarrow \Hom_{U(L)}(P_1,M) \rightarrow
\Hom_{U(L)}(P_2,M) \rightarrow \dots
$$
of $L$-modules coincides with the Chevalley--Eilenberg cohomology 
$\Homol^\bullet(L,M)$, $L$ has cohomological dimension $n$ if and only if there
is an $L$-module $M$ such that $\Homol^n(L,M) \ne 0$, and one of the following 
equivalent conditions holds:
\begin{enumerate}[\upshape(i)]
\item
$\Homol^i(L,M) = 0$ for any $L$-module $M$, and any $i > n$;
\item
$\Homol^{n+1}(L,M) = 0$ for any $L$-module $M$.
\end{enumerate}

A similar notion may be defined for other classes of algebraic systems with good
cohomology theories, for example, for groups and associative algebras. 

The Shapiro lemma on the cohomology of a coinduced module implies that if $S$ is
a subalgebra of a Lie algebra $L$, then $\cd(S) \le \cd(L)$. As cohomological 
dimension of the one-dimensional Lie algebra is equal to one, the cohomological
dimension of any nonzero Lie algebra is $\ge 1$. In particular, the class of Lie
algebras of cohomological dimension one is closed with respect to subalgebras.

Due to the standard interpretation of the second cohomology, the condition for
a Lie algebra $L$ to be of cohomological dimension one is equivalent to the 
condition that each short exact sequence 
$$
0 \to M \to \cdot \to L \to 0
$$ 
of Lie algebras, where the $L$-module $M$ is considered as an abelian Lie algebra, splits. The latter condition holds for a free Lie algebra, due to
its universal property, and hence a free Lie algebra (of any rank) has 
cohomological dimension one. The same is true for free groups and free 
associative algebras.

The celebrated Stallings--Swan theorem says that for groups the converse is 
true: a group of cohomological dimension one is free (see, for example, 
\cite{lecturenotes}). A question by Bourbaki 
(\cite[Chapitre II, \S 2, footnote to Exercice 9]{bourbaki}) asks whether 
the same is true for Lie algebras, i.e., whether a Lie algebra of cohomological
dimension one is free.

Feldman \cite{feldman} answered this question affirmatively in the case of 
$2$-generated Lie algebras. For a while, it was widely believed that the answer
is affirmative in general (the author has witnessed several attempts of the 
proof\footnote{
And, curiously enough, the claim about freeness of a Lie algebra of 
cohomological dimension one is formulated as an Exercise 7.6.3 in the book 
\cite{weibel}.
}), 
until Mikhalev, Umirbaev and Zolotykh constructed an example of a non-free Lie 
algebra of cohomological dimension one over a field of characteristic $>2$ (see
\cite{muz}; note that the cases of characteristic zero and characteristic $2$ 
remain widely open). This example is not a $p$-algebra, and in the same paper they made the following conjecture: a Lie $p$-algebra of 
cohomological dimension one is a free Lie $p$-algebra 
(\cite[Conjecture 2]{muz}). As stated, the conjecture is somewhat misleading, 
for a free Lie $p$-algebra is not of cohomological dimension one: its cohomological dimension is equal to infinity. 
Indeed, for any nonzero element $x$ of such an algebra, the elements 
$x, x^{[p]}, x^{[p]^2}, \dots$ span an infinite-dimensional abelian subalgebra,
whose cohomological dimension is equal to infinity (see Lemma \ref{lemma-cd} 
below).

This conjecture may be repaired in two ways. First, one may merely ask for a
description of Lie $p$-algebras of cohomological dimension one. A (trivial) 
answer to this question is given in \S \ref{sec-cd1}: such algebras are 
one-dimensional. Another possibility is to replace cohomological dimension with
\emph{restricted} cohomological dimension; this is discussed in \S \ref{sec-p}.
Also, \S \ref{sec-cd1} contains auxiliary results and conjectures related to 
the old Jacobson conjecture about periodic Lie $p$-algebras, and in
\S \ref{sec-dim1} we discuss the problem of description of Lie $p$-algebras 
without $2$-dimensional subalgebras.

The ground field $K$ is arbitrary, unless stated otherwise. The direct sum, 
$\oplus$, is always understood as the direct sum of vector spaces.

\section{
Lie $p$-algebras of cohomological dimension one and almost-periodic algebras
}\label{sec-cd1}

The following lemma is elementary (and, undoubtedly, well known) but useful.

\begin{lemma}\label{lemma-cd}\hfill
\begin{enumerate}[\upshape(i)]
\item
The cohomological dimension of an abelian Lie algebra (finite- or 
infinite-dimensional) is equal to its dimension.
\item
The cohomological dimension of a finite-dimensional Lie algebra is equal to its
dimension.
\end{enumerate}
\end{lemma}

\begin{proof}
(i)
It is clear that the cohomological dimension of a Lie algebra does not exceed 
its dimension. For an abelian Lie algebra $L$, we have 
$\Homol^n(L,K) = (\bigwedge^n L)^\star$ for any $n$ (where $^\star$ denotes the
dual vector space).

(ii)
This follows from the fact that the global dimension of the universal enveloping
algebra of an $n$-dimensional Lie algebra is equal to $n$ (see, for example, 
\cite[Exercise 7.7.2]{weibel}). 

Alternatively, one may employ the same scheme as used in the proof of the fact 
(actually, a long-standing conjecture of Seligman) that for any $n$-dimensional
Lie algebra $L$ over a field of positive characteristic, and any 
$0 \le k \le n$, there is a \emph{finite-dimensional} $L$-module $M$ such that 
$\Homol^k(L,M) \ne 0$ (\cite[\S 2]{dzhu} or \cite[Corollary 2.2]{fs}). The 
required modules $M$ are constructed as induced or coinduced modules over 
truncated, and hence finite-dimensional, versions of the universal enveloping 
algebra $U(L)$. If we replace those truncated versions by the ordinary universal
enveloping algebra, we will get the proof of the same statement valid over an 
arbitrary field (but, of course, the modules $M$ with non-vanishing cohomology 
will be no longer finite-dimensional).
\end{proof}

\begin{corollary}\label{cor-1}
A Lie algebra of cohomological dimension one does not contain a $2$-dimensional
subalgebra.
\end{corollary}

\begin{theorem}\label{th-1}
A Lie $p$-algebra of cohomological dimension one is one-dimensional.
\end{theorem}

\begin{proof}
Let $L$ be a Lie $p$-algebra of cohomological dimension one. For any $x\in L$, 
we have $[x,x^{[p]}] = 0$, and by Corollary \ref{cor-1},
\begin{equation}\label{eq-p1}
x^{[p]} = \lambda(x) x
\end{equation}
for some $\lambda(x) \in K$. 

Suppose $L$ is of dimension $>1$, and pick two linearly independent elements
$x,y \in L$. By \cite{feldman}, the subalgebra of $L$ generated by $x,y$ is 
free. But according to (\ref{eq-p1}), $(\ad x)^p(y) = \lambda(x)[y,x]$, a 
contradiction.
\end{proof}

Let us now reflect on the condition (\ref{eq-p1}). This condition is reminiscent
of various conditions on the $p$-map studied by Jacobson and others. The major 
open problem in this area is the conjecture of Jacobson that a periodic Lie 
$p$-algebra is abelian (see \cite[Chapter V, Exercise 16]{jacobson}). Recall 
that a Lie algebra $L$ is called \emph{periodic} if for any $x\in L$ there is 
integer $n(x)>0$ such that $x^{[p]^{n(x)}} = x$. The strongest result toward 
this conjecture is due to Premet: a periodic \emph{finite-dimensional} Lie 
algebra is abelian (\cite[Corollary 1]{premet}).

Generalizing the condition of periodicity, let us call a Lie $p$-algebra $L$
\emph{almost periodic}, if for any $x\in L$, there are an integer $n(x)>0$ and
an element $\lambda(x) \in K$ such that
\begin{equation}\label{eq-x-th}
x^{[p]^{n(x)}} = \lambda(x) x .
\end{equation}
The elements for which $\lambda(x) = 0$, i.e., $x^{[p]^{n(x)}} = 0$, will be 
called \emph{$p$-nilpotent}.

\begin{proposition}\label{prop-jacobson}
Let $L$ be an almost periodic Lie $p$-algebra of dimension $>1$ over an 
algebraically closed field, with $n(x)$'s uniformly bounded. Then $L$ contains 
a nonzero $p$-nilpotent element.
\end{proposition}

Note some other related results connecting properties of Lie ($p$-)algebras and
its elements:
\begin{enumerate}[\upshape(i)]
\item
Latyshev in \cite[Proof of Lemma 1]{latyshev} provides interesting reasonings
related to the condition (\ref{eq-p1}).
\item
Chwe proved in \cite{chwe} (see also \cite[Theorem 3.10]{strade-farn}) that a 
Lie $p$-algebra over an algebraically closed field with a nondegenerate $p$-map
is abelian. 
\item
Farnsteiner investigated in \cite{farns-semilin} Lie $p$-algebras in which some
power $[p]^n$ of the $p$-map is $p^n$-semilinear. The condition (\ref{eq-x-th})
is somewhat reminiscent of semilinearity (in some sense stronger, in some sense
weaker). 
\item It is well known that any finite-dimensional Lie algebra over an 
algebraically closed field contains a nilpotent element. (For Lie $p$-algebras,
this follows from the Seligman--Jordan--Chevalley decomposition -- see, for 
example, \cite[Proof of Theorem 3]{premet}, and for a short elementary proof 
valid for arbitrary Lie algebras, see \cite{benkart-isaacs}). 
Proposition \ref{prop-jacobson} establishes a similar result for not necessarily
finite-dimensional Lie algebras, but subject to the strong condition of 
uniformly bounded $p$-periodicity.
\end{enumerate}

Note also that the condition of the ground field being algebraically closed 
cannot be dropped from the proposition, for any nonsplit $3$-dimensional simple
Lie algebra over a field of characteristic $p>0$ provides a counterexample: it 
satisfies the condition $x^{[p]} = \lambda(x) x$ for any nonzero element $x$, 
but does not have nonzero $p$-nilpotent elements (i.e., $\lambda(x) \ne 0$ for 
any $x\ne 0$).

\begin{proof}[Proof of Proposition \ref{prop-jacobson}]
Since $n(x)$'s are uniformly bounded, we may assume that
\begin{equation}\label{eq-x}
x^{[p]^n} = \lambda(x) x
\end{equation}
for some fixed $n$ (for example, by letting $n$ to be the product of all 
distinct $n(x)$'s, and redenoting $\lambda(x)$'s appropriately).

Pick any two linearly independent elements $x,y\in L$, and set 
$\varphi_{xy}(t) = \lambda(x + ty)$, for $t\in K$. Using the well known Jacobson
binomial formula for the $p$-map (strictly speaking, its generalization for the
$n$th power of the $p$-map -- see, for example, \cite[\S 1]{farns-semilin}), we
have
\begin{equation}\label{eq-pn}
\varphi_{xy}(t)(x + ty) = (x + ty)^{[p]^n} = 
x^{[p]^n} + t^{p^n}y^{[p]^n} + \sum_{i=1}^{p^n-1} t^i s_i(x,y) =
\lambda(x)x + t^{p^n}\lambda(y)y + \sum_{i=1}^{p^n-1} t^i s_i(x,y) , 
\end{equation}
where $s_i(x,y)$ are certain Lie monomials in $x,y$. Completing $x,y$ to a basis
of $L$, writing the $s_i(x,y)$'s as linear combinations of basis elements, and 
collecting all coefficients of $x$ in (\ref{eq-pn}), we get that 
$\varphi_{xy}(t)$ is a polynomial in $t$ with the constant term $\lambda(x)$.

Suppose that there is a pair $x,y$ such that $\varphi_{xy}(t)$ is not constant. 
Since the ground field $K$ is algebraically closed, $\varphi_{xy}(t)$ has a
root $\xi$. This means that the nonzero element $x + \xi y$ is nilpotent.

Suppose now that for any pair $x,y\in L$, $\varphi_{xy}(t)$ is constant, i.e.,
$\varphi_{xy}(t) = \lambda(x)$. This means that $\lambda(x + ty) = \lambda(x)$ 
for any linearly independent $x,y\in L$, and any $t\in K$, and, consequently, 
$\lambda(x) = \lambda$ is constant. If $\lambda \ne 0$, then substituting in 
(\ref{eq-x}) $\alpha x$ instead of $x$, we get that $\alpha^{p^n} = \alpha$ 
for any $\alpha \in K$, i.e., $K$ is a finite field, a contradiction. Hence 
$\lambda = 0$, and every element of $L$ is nilpotent. 
\end{proof}

Proposition \ref{prop-jacobson} can be used to give an alternative proof of 
Theorem \ref{th-1}, not utilizing the Feldman result about $2$-generated Lie 
algebras. For that, we need another couple of elementary lemmas.

\begin{lemma}\label{cor-nilp}
Let $x,y$ be two elements of a Lie algebra without $2$-dimensional subalgebras,
such that $(\ad x)^n y = 0$ for some $n$. Then $x,y$ are linearly dependent.
\end{lemma}

\begin{proof}
Applying repeatedly the condition of absence of $2$-dimensional subalgebras, we
can lower the degree $n$. Indeed, $(\ad x)^n y = [(\ad x)^{n-1}(y),x] = 0$ 
implies $[(\ad x)^{n-2}(y),x] = (\ad x)^{n-1}(y) = \lambda x$ for some 
$\lambda \in K$, which, in turn, implies $\lambda = 0$. Repeating this process, 
we get eventually $[y,x] = 0$, and hence $x,y$ are linearly dependent.
\end{proof}

\begin{lemma}\label{lemma-2dim}
A Lie $p$-algebra of dimension $>1$ over an algebraically closed field contains
a $2$-di\-men\-si\-o\-nal subalgebra.
\end{lemma}

This lemma may be considered as an extension of the elementary fact that
a finite-dimensional Lie algebra of dimension $>1$ over an algebraically closed
field contains a $2$-dimensional subalgebra (for the situation over 
non-algebraically closed fields, see the next section). We do not assume 
finite-dimensionality, but the presence of a $p$-structure is the condition 
strong enough to infer the same conclusion.

\begin{proof}
Let $L$ be a Lie $p$-algebra without two-dimensional subalgebras. By the same
reasoning as in the proof of Theorem \ref{th-1}, $L$ satisfies the condition 
(\ref{eq-p1}). According to Proposition \ref{prop-jacobson} (with $n(x) = 1$ for
all $x$), $L$ is either one-dimensional, or contains a nonzero nilpotent 
element. In the latter case by Lemma \ref{cor-nilp}, $L$ is one-dimensional too,
a contradiction.
\end{proof}

\begin{proof}[An alternative proof of Theorem \ref{th-1}]
Let $L$ be a Lie $p$-algebra of cohomological dimension one over a field $K$. 
Consider the Lie algebra $\overline L = L \otimes_K \overline K$ over the 
algebraic closure $\overline K$ of $K$. Any $\overline L$-module $M$ is also an 
$L$-module, and by the change-of-ring theorem (see, for example, 
\cite[Theorem 9.1.7]{weibel}), 
$\Homol^2(\overline L,M) \simeq \Homol^2(L,M) = 0$, hence $\overline L$ is a Lie
$p$-algebra of cohomological dimensional one over an algebraically closed 
field. Now Lemma \ref{lemma-2dim} and Corollary \ref{cor-1} imply that 
$\overline L$, and hence $L$, is one-dimensional.
\end{proof}

\section{Digression: Lie $p$-algebras without $2$-dimensional subalgebras
}\label{sec-dim1}

Would it be possible in the latter proof to avoid passage to the algebraic 
closure of the ground field and reference to the change-of-ring theorem, and 
argue directly about algebras over an arbitrary field? Note that the condition of algebraic 
closedness cannot be removed from Lemma \ref{lemma-2dim}, by virtue of the same counterexample as to Proposition \ref{prop-jacobson}: a 
nonsplit $3$-dimensional simple Lie algebra. This is, however, together with
another $7$-dimensional algebra peculiar to characteristic $3$, is the only 
finite-dimensional counterexample, as the following proposition shows.

\begin{proposition}\label{prop-2}\hfill
\begin{enumerate}[\upshape(i)]
\item
A finite-dimensional Lie $p$-algebra of dimension $>1$ over a field of 
characteristic $p>3$ does not contain $2$-di\-men\-si\-o\-nal subalgebras if and
only if it is a nonsplit $3$-dimensional simple algebra.
\item
A finite-dimensional Lie $3$-algebra of dimension $>1$ over a field of 
characteristic $3$ does not contain $2$-dimensional subalgebras if and only if 
it is either a nonsplit $3$-dimensional simple algebra, or a nonsplit form of 
$psl(3)$.
\item
A finite-dimensional Lie $2$-algebra of dimension $>1$ over a field of 
characteristic $2$ contains a $2$-dimensional subalgebra. 
\end{enumerate}
\end{proposition}

To prove this proposition, we will need the following elementary lemmas.

\begin{lemma}\label{lemma-kx}
Let $L$ be a finite-dimensional Lie algebra without $2$-dimensional subalgebras.
Then for any nonzero $x\in L$, $L = Kx \oplus [L,x]$.
\end{lemma}

\begin{proof}
The statement is vacuous if $L$ is $1$-dimensional, so assume $\dim L > 1$. As the kernel of the linear map $\ad x$ is $1$-dimensional, its image $[L,x]$
is of codimension $1$ in $L$. If $x \in [L,x]$, then $x = [y,x]$ for some 
$y \in L$, and $Kx \oplus Ky$ is a $2$-dimensional subalgebra, a contradiction.
\end{proof}

The following lemma can be deduced from \cite[Proposition 3.2]{varea87}. We give
a direct short proof.

\begin{lemma}\label{lemma-l}
A finite-dimensional Lie algebra of dimension $>1$ without $2$-dimensional 
subalgebras is simple.
\end{lemma}

\begin{proof}
Let $L$ be a finite-dimensional Lie algebra of dimension $>1$ without 
$2$-dimensional subalgebras, and $I$ is a proper ideal in $L$. Then either $I$ 
is one-dimensional, or $I$ also does not have $2$-dimensional subalgebras. In 
the former case write $I = Kx$ for some $x\in I$; then for any 
$y \in L \backslash I$, $Kx \oplus Ky$ is a $2$-dimensional subalgebra, a contradiction. 

In the latter case, take a nonzero $x\in I$, and a nonzero $y\in L\backslash I$.
Since $[y,x] \in I$, by Lemma \ref{lemma-kx} there are $\lambda \in K$ and 
$a\in I$ such that $[y,x] = \lambda x + [a,x]$. But then $K(y-a) \oplus Kx$ is a
$2$-dimensional subalgebra in $L$, a contradiction.
\end{proof}

\begin{proof}[Proof of Proposition \ref{prop-2}]
We merely push a bit further arguments from \cite[\S 1]{latyshev}. Let $L$ be a
finite-dimensional Lie $p$-algebra of dimension $> 1$ without $2$-dimensional
subalgebras. By the same reasoning as in Lemma \ref{lemma-kx}, for any nonzero 
$h\in L$, $Kh$ coincides with its own normalizer, and hence is a Cartan 
subalgebra of $L$. For any nonzero $x\in L$, and any nonzero element $\gamma$ from the centroid of $L$,
the elements $x$ and $\gamma x$ are commuting, and hence are linearly dependent
over the base field $K$. Therefore, the centroid of $L$ coincides with $K$, and
since by Lemma \ref{lemma-l} $L$ is simple, $L$ is central simple. In 
particular, $\overline L = L \otimes_K \overline K$ is a simple Lie algebra over the 
algebraic closure $\overline K$ of $K$, and $h \otimes_K \overline K$ is an 
one-dimensional Cartan subalgebra in $\overline L$, i.e., $\overline L$
is a simple Lie $p$-algebra of rank one. 

If $p>3$, then by \cite[Theorem 3]{kaplansky}, $\overline L$ is isomorphic 
either to $sl(2)$, or to the ($p$-dimensional) Witt algebra. But each form of 
the Witt algebra is isomorphic to the derivation algebra of 
$K[x]/(x^p - \lambda 1)$ for some $\lambda \in K$ (see, for example, 
\cite[Proposition 3 and Corollary 1]{waterhouse}) and hence contains a 
$2$-dimensional nonabelian subalgebra (for example, spanned by the elements 
$\frac{\dcobound}{\dcobound x}$ and $x\frac{\dcobound}{\dcobound x}$), a 
contradiction. Hence $L$ is a form of $sl(2)$, obviously nonsplit. Conversely,
all proper subalgebras of a nonsplit form of $sl(2)$ are one-dimensional.

Let $p=3$. As in the proof of Theorem \ref{th-1}, $L$ satisfies the condition
(\ref{eq-p1}), and, arguing as in the proof of \cite[Theorem 3]{kaplansky}, over
$\overline K$ we may rescale $h$ to get an element $h^\prime$ such that 
${h^\prime}^{[3]} = h^\prime$. Hence $\overline K h^\prime$ is an 
one-dimensional Cartan subalgebra of $\overline L$, with roots lying in the 
prime subfield $GF(3)$, and by \cite[Theorem 9]{kaplansky}, $\overline L$ is 
isomorphic either to $sl(2)$, or to ($7$-dimensional) $psl(3)$. Each form of 
$psl(3)$ is isomorphic to the quotient $\mathbb O^{(-)}/K1$ for an octonion 
algebra $\mathbb O$ over $K$ (considered as a Lie algebra with respect to 
commutator $[a,b] = ab-ba$), by the $1$-dimensional center $K 1$ (see, for 
example, \cite[Theorem 4.26]{elduque-kochetov}). If $\mathbb O$ is nonsplit, 
then each pair of commuting elements in $\mathbb O$ generates a nonsplit 
quaternion algebra (see, for example, \cite[proof of Theorem 3.17]{schafer}),
and hence each proper Lie subalgebra of $\mathbb O^{(-)}/K1$ is either 
one-dimensional, or nonsplit $3$-dimensional simple (alternatively, we may
employ again Lemma \ref{lemma-l} and \cite[Theorem 9]{kaplansky} to reach the 
same conclusion; these $7$-dimensional algebras also appear, with a different
matrix realization, in \cite[Example 2]{gein}).

If $p=2$, the same reasonings as in the case $p=3$ show that $\overline L$ has 
an one-dimensional Cartan subalgebra with roots lying in the prime subfield 
$GF(2)$. Thus there are just two roots, $0$ and $1$, and by 
\cite[Theorem 8]{kaplansky} the root space $\overline L_1$ is $2$-dimensional, 
so $\overline L$ is $3$-dimensional. But then it is easy to see that 
$\overline L$ is not a $p$-algebra (its $p$-envelope is $5$-dimensional), hence
neither is $L$, a contradiction.
\end{proof}

Over perfect fields of characteristic $\ne 2,3$, a result similar to 
Proposition \ref{prop-2}(i) holds for any finite-dimensional Lie algebra, not 
necessary $p$-algebra. Indeed, by the results of Premet (\cite{premet-bssr} 
and references therein), any such Lie algebra $L$ either contains a noncentral 
element $x$ such that $(\ad x)^2 = 0$, or satisfies 
$\mathfrak g \subseteq L \subseteq \Der(\mathfrak g)$ for some form of a 
classical simple Lie algebra $\mathfrak g$. If $L$ does not contain
$2$-dimensional subalgebras, Lemma \ref{cor-nilp} rules out the first 
possibility, and the second possibility together with Lemma \ref{lemma-l} 
implies that $L = \mathfrak g$ is a form of a classical simple Lie algebra. But
since $L$ is of rank $1$, $L$ is $3$-dimensional.

Over arbitrary, not necessary perfect, fields, where the Galois-cohomological
machinery used in \cite{premet-bssr} is not available, arguments similar to, but
more involved than those used in the proof of Proposition \ref{prop-2}, and
addressing subtle questions such as the existence of Cartan subalgebras in a 
Lie algebra over an arbitrary field, and structure of forms of 
Albert--Zassenhaus algebras, may be used to get a description of arbitrary
finite-dimensional Lie algebras with even more weaker restrictions on 
subalgebras (for example, Lie algebras all whose abelian subalgebras are 
one-dimensional). This, however, will lead us too far away from the current 
topic and, hopefully, will be treated elsewhere.

Returning to Lie $p$-algebras without $2$-dimensional subalgebras: the 
infinite-dimensional situation is, naturally, much more difficult: all what we 
can do is to make the following

\begin{conjecture}\label{conj-3}
An infinite-dimensional Lie $p$-algebra contains a $2$-di\-men\-si\-o\-nal 
subalgebra.
\end{conjecture}

It is an intriguing open question whether there exist infinite-dimensional Lie 
algebras all whose proper subalgebras are one-dimensional (Lie-algebraic analogs
of Tarski's monsters in group theory constructed by Olshanskii). 
Conjecture \ref{conj-3} implies that there are no such algebras in the class of
Lie $p$-algebras.

\section{
Lie $p$-algebras of restricted cohomological dimension one
}\label{sec-p}

When speaking about cohomological dimension, we consider the category of 
\emph{all} Lie algebra modules, including infinite-dimensional ones. If we 
restrict ourselves to, say, finite-dimensional Lie algebras and the category 
of finite-dimensional modules, the whole subject, both in results and methods 
employed, becomes quite different. In fact, we cannot longer speak about 
cohomological dimension, as vanishing of the cohomology in a given degree does 
not imply vanishing in higher degrees. A sample of results in this domain: in 
characteristic zero, an ``almost'' converse of the classical Whitehead Lemmas 
holds (\cite{z1}, \cite{z2}), and in positive characteristic, the existence of a
finite-dimensional module with non-vanishing cohomology in any degree not 
greater than the dimension of the algebra, which was already mentioned in the 
proof of Lemma \ref{lemma-cd}.

Still, instead of the category of all modules with ordinary cohomology, we can 
consider a smaller subcategory of modules with a good-behaving cohomology 
theory: for example, the category of all restricted modules with restricted 
cohomology. Recall that for a Lie $p$-algebra $L$, and a bimodule $M$ over its 
restricted universal enveloping algebra $u(L)$, we have 
\begin{equation}\label{eq-isom}
\Homol^n_*(L, M^{\ad}) \simeq \HH^n(u(L), M) ,
\end{equation}
where $\Homol_*$ and $\HH$ stand for the restricted cohomology of a Lie 
$p$-algebra, and the Hochschild cohomology of an associative algebra, 
respectively, and $M^{\ad}$ is the restricted $L$-module structure on $M$ 
defined via $x \bullet m = xm - mx$ for $x\in L$, $m \in M$.

The definition of the \emph{restricted cohomological dimension} of $L$ 
(notation: $\cd_*(L)$) repeats the definition of the ordinary cohomological 
dimension, with projective resolutions (\ref{eq-resol}) considered in the 
category of left (or right) modules over $u(L)$.

As in the unrestricted case, Shapiro's lemma for restricted cohomology implies
that the restricted cohomological dimension does not increase when passing to
subalgebras. In particular, a subalgebra of a Lie $p$-algebra of restricted 
cohomological dimension one is of restricted cohomological dimension one or 
zero. A free Lie $p$-algebra has restricted cohomological dimension one. 

Recall that an element $x$ of a Lie $p$-algebra $L$ is called \emph{semisimple},
if $x$ is a linear combination of its $p$-powers $x^{[p]^k}$, $k=1,2,\dots$. If
$L$ is a finite-dimensional \emph{torus}, i.e., an abelian Lie $p$-algebra 
consisting of semisimple elements, then $u(L)$ is a commutative semisimple 
algebra, and hence $\cd_*(L) = 0$. Conversely, the main theorem of 
\cite{hochsch} (see also \cite[Satz 10]{strade} and 
\cite[Theorem 3.1]{farns-coho}) amounts to saying (in a different terminology) 
that any finite-dimensional Lie $p$-algebra of restricted cohomological 
dimension zero is a torus. Moreover, according to \cite[Theorem 9.2.11]{weibel},
restricted cohomological dimension zero implies finite-dimensionality, so 
Lie $p$-algebras of restricted cohomological dimension zero are exactly 
finite-dimensional tori.

The existence of nontrivial Lie $p$-algebras of restricted cohomological 
dimension zero makes the situation somewhat similar to the associative one: due
to the classical results of Eilenberg, Hochschild, Rosenberg, Zelinsky, and 
others obtained in the 1940--1950s, it is known that associative algebras of 
cohomological dimension zero are exactly finite-dimensional separable algebras 
(see, for example, \cite[Theorem 9.2.11]{weibel}). As for associative algebras 
of cohomological dimension one, the multitude of examples of such algebras 
seemingly defies any attempt to classify them (see \cite{cuntz-quillen}). One 
may hope that the Lie $p$-algebraic situation is more tame (see 
Conjecture \ref{conj} below).

Also, similarly to the associative case, the existence of nontrivial Lie 
$p$-algebras of restricted cohomological dimension zero allows, by the extension
procedure, to get new Lie $p$-algebras of restricted cohomological dimension one
from already known ones. 

\begin{lemma}\label{lemma-ext}
Let $I$ be a $p$-ideal of a Lie $p$-algebra $L$. Then:
\begin{enumerate}[\upshape(i)]
\item $\cd_*(L) \le \cd_*(I) + \cd_*(L/I)$;
\item if $\cd_*(I) = 0$, then $\cd_*(L) = \cd_*(L/I) $;
\item if $\cd_*(L/I) = 0$, then $\cd_*(L) = \cd_*(I) $.
\end{enumerate}
\end{lemma}

Part (i) is a restricted analogue of \cite[Theorem 3.11.9]{bokut-kukin}.

\begin{proof}
This follows immediately from the Lyndon--Hochschild--Serre spectral sequence
converging to $\HH^{s+t}(u(L),M)$ and having the $\E_2$ term
$$
\E_2^{st} = \HH^s(u(L/I), \HH^t(u(I),M))
$$
(here $M$ is an arbitrary $u(L)$-module). If $\HH^t(u(I),M) = 0$ for any 
$t > n$, and $\HH^s(u(L/I),M) = 0$ for any $s > m$, then $\E_2^{st} = 0$ for
any $s + t > n+m$, and hence $\HH^i(u(L),M) = 0$ for any $i > n+m$, what proves
(i).

To prove (ii), note that $\cd_*(I) = 0$ implies that the only non-vanishing 
$\E_2$ terms are $\E_2^{s0}$, the spectral sequence stabilizes at $\E_2$, and
$\HH^n(u(L),M) \simeq \E_2^{n0} = \HH^n(u(L/I), M^I)$. Since any restricted 
$L/I$-module can be lifted to a restricted $L$-module by letting $I$ act 
trivially, the desired equality follows.

Part (iii) follows from (i): $\cd_*(L/I) = 0$ implies $\cd_*(L) \le \cd_*(I)$. 
On the other hand, $\cd_*(I) \le \cd_*(L)$, and the desired equality follows.
\end{proof}

Parts (ii) and (iii) of Lemma \ref{lemma-ext} show in particular, that extending
a Lie $p$-algebra of restricted cohomological dimension zero by a Lie 
$p$-algebra of restricted cohomological dimension one, or, vice versa, extending
a Lie $p$-algebra of restricted cohomological dimension one by a Lie $p$-algebra of 
restricted cohomological dimension zero, we get a Lie $p$-algebra of 
cohomological dimension one. As any extension of a Lie $p$-algebra of restricted
cohomological dimension $\le 1$ splits, any algebra which can be obtained 
starting from a free Lie $p$-algebra by successively applying such extensions, 
has the following form:
\begin{equation}\label{eq-iter}
( \dots ((\mathscr L \bowtie T_1) \bowtie T_2) \dots ) \bowtie T_n ,
\end{equation}
where $\mathscr L$ is a free Lie $p$-algebra, $T_1, \dots, T_n$ are 
finite-dimensional tori, and each symbol $\bowtie$ stands either for 
$\rtimes$ (action of the left-hand side on the right-hand side), or for 
$\ltimes$ (action of the right-hand side on the left-hand side).

\begin{conjecture}\label{conj}
Any Lie $p$-algebra of restricted cohomological dimension one is of the form 
(\ref{eq-iter}).
\end{conjecture}

In particular, this conjecture implies that a Lie $p$-algebra of restricted 
cohomological one has a free Lie $p$-subalgebra of finite codimension.

Let us establish some facts about Lie $p$-algebras of restricted cohomological
dimension one, providing a (limited) evidence in support of the conjecture.

The following fact was established in \cite[Theorem 5.1]{chwe-tams} using a
not entirely trivial result from homological algebra due to Kaplansky, and 
independently in \cite[Remark 2]{berkson}, using previous results of Eilenberg 
and Nakayama. We give an alternative, and arguably a more elementary proof -- a
mere reformulation of known (and easy) results on the cohomology of commutative 
associative algebras.

\begin{lemma}\label{lemma-fin}
The restricted cohomological dimension of a finite-dimensional Lie $p$-algebra 
is either zero or infinity.
\end{lemma}

\begin{proof}
Let $L$ be a finite-dimensional Lie algebra of restricted cohomological 
dimension $>0$, i.e. not a torus. Since $L$ is not a torus, there is $x\in L$ 
satisfying the relation  of the form
\begin{equation}\label{eq-f}
\lambda_1 x^{[p]} + \lambda_2 x^{[p]^2} + \dots + \lambda_n x^{[p]^n} = 0
\end{equation}
for some $n \ge 1$ and $\lambda_1, \lambda_2, \dots, \lambda_n \in K$, 
$\lambda_n \ne 0$. For the $p$-subalgebra $(x)_p$ generated by $x$, we have
$u((x)_p) \simeq K[x]/(f)$, where the polynomial $f$ is obtained from the 
left-hand side of (\ref{eq-f}) by replacing $p$-powers in the Lie $p$-algebra by
the ordinary $p$-powers in the polynomial algebra: 
$f(t) = \lambda_1 t^p + \lambda_2 t^{p^2} + \dots + \lambda_n t^{p^n}$.
 
The Hochschild cohomology of quotients of polynomial algebras is well 
understood -- see, for example, \cite{holm} and references therein. In 
particular, in \cite[Proposition 2.2]{holm} a periodic free resolution of such algebras is 
constructed, from which it follows that the complex computing the Hochschild
cohomology of $K[x]/(f)$ is of the form
$$
K[x]/(f) \overset{0}{\longrightarrow} 
K[x]/(f) \overset{f^\prime}{\longrightarrow} 
K[x]/(f) \overset{0}{\longrightarrow} 
K[x]/(f) \overset{f^\prime}{\longrightarrow} 
\dots
$$
Since $f^\prime$ (the formal derivative of $f$) vanishes, 
$\HH^n(K[x]/(f), K[x]/(f))$ does not vanish for any $n$. As $K[x]/(f)$ is 
commutative, $K[x]/(f)^{\ad}$, as an $(x)_p$-module, is the direct sum of $p^n$
copies of the trivial $(x)_p$-module $K$, and due to the isomorphism 
(\ref{eq-isom}), $\Homol^n_*((x)_p,K)$ is nonzero for any $n$. Consequently, the
restricted cohomological dimension of $(x)_p$, and thus of $L$, is equal to 
infinity.
\end{proof}

\begin{proposition}\label{prop-1}
A $p$-subalgebra of a Lie $p$-algebra of finite restricted cohomological 
dimension is either a torus, or is infinite-dimensional.
\end{proposition}

\begin{proof}
Follows from Lemma \ref{lemma-fin}.
\end{proof}

In particular, in a Lie $p$-algebra $L$ of finite restricted cohomological 
dimension, every nonzero $p$-algebraic element (i.e., on which a certain 
$p$-polynomial vanishes) is $p$-semisimple. This can be considered as a 
Lie-$p$-algebraic analog of the well known fact that groups of finite 
cohomological dimension are torsion-free (see, for example, 
\cite[p.~6, Corollary 2]{lecturenotes}).

\begin{proposition}
An abelian $p$-subalgebra of a Lie $p$-algebra of restricted cohomological
dimension one is either a torus, or is isomorphic to the direct sum of a torus 
and the free Lie $p$-algebra of rank one.
\end{proposition}

\begin{proof}
Let $L$ be an abelian subalgebra of a Lie $p$-algebra of restricted 
cohomological dimension one. The restricted cohomological dimension of $L$ is 
either equal to zero, in which case $L$ is a torus, or is equal to one. In the 
latter case, assume first that $L$ has no nonzero $p$-algebraic elements.

To prove that $L$ is a free Lie $p$-algebra of rank one, it is enough to prove 
that any two commuting elements of $L$, say, $x$ and $y$, can be represented as
$p$-polynomials of a third element. Suppose the contrary. By 
Proposition \ref{prop-1}, each of $x$, $y$ generates the free Lie $p$-algebra of
rank one, and hence the restricted universal enveloping algebra of the 
$p$-subalgebra $S$ of $L$ generated by $x,y$, is isomorphic to the polynomial 
algebra in two variables $K[x,y]$. The latter algebra has non-vanishing 2nd 
Hochschild cohomology (for example, 
$\HH^2(K[x,y],K[x,y]) \simeq \bigwedge^2 (\Der(K[x,y])) \otimes_{K[x,y]} K$ by 
the Hochschild--Kostant--Rosenberg theorem), and by reasoning as at the end of 
the proof of Lemma \ref{lemma-fin}, we get that $\Homol^2_*(S,K)$ does not 
vanish, whence $\cd_*(L) \ge \cd_*(S) \ge 2$, a contradiction.

In the general case, consider the set $T$ of all $p$-semisimple elements of $L$.
Obviously, $T$ forms a proper subalgebra, and hence a proper ideal, of $L$. By 
Lemma \ref{lemma-ext}(ii), the quotient $L/T$ is an abelian Lie $p$-algebra of 
restricted cohomological dimension one. Since $L/T$ does not have nonzero
$p$-algebraic elements, $L/T$ is isomorphic to the free Lie $p$-algebra of rank
one by above. The extension obviously splits, and the desired conclusion 
follows.
\end{proof}

The next lemma shows that the (ordinary) cohomology of Lie $p$-algebras of 
restricted cohomological dimension one behaves in a rather peculiar way.

\begin{lemma}\label{prop-3}
Let $L$ be a Lie $p$-algebra of restricted cohomological dimension one, and $M$
a restricted $L$-module. Then
\begin{equation}\label{eq-ss}
\Homol^n(L,M) \simeq 
\Hom\Big(\bigwedge\nolimits^n L,\> M^L\Big) \oplus 
\Hom\Big(\bigwedge\nolimits^{n-1} L,\> \Homol^1_*(L,M)\Big)
\end{equation}
for any $n \ge 1$.
\end{lemma}

Here $\Hom(\bigwedge^n V, W)$ denotes, as usual, the space of skew-symmetric 
$n$-linear maps from one vector space to another.

\begin{proof}
This follows from a particular form of the Grothendieck spectral sequence 
relating restricted and ordinary cohomology. Namely, for a Lie $p$-algebra $L$
and a restricted $L$-module $M$, there is a spectral sequence with the $\E_2$ 
term 
$$
\E_2^{st} \simeq \Hom\Big(\bigwedge\nolimits^t L,\> \Homol^s_*(L,M)\Big)
$$ 
converging to $\Homol^{s+t}(L,M)$ (see \cite[Theorem 4.1]{farns-coho}, and
also \cite[Proposition 5.3]{fp} and \cite[Corollary 1.3]{muzere}). 

If $\Homol^s_*(L,M) = 0$ for $s \ge 2$, the only nonvanishing $\E_2$ terms are
$\E_2^{0t}$ and $\E_2^{1t}$. Hence the spectral sequence stabilizes at $\E_2$, 
$\Homol^n(L,M) \simeq \E_2^{0n} \oplus \E_2^{1,n-1}$ for any $n \ge 1$, and 
(\ref{eq-ss}) follows.
\end{proof}

Lemma \ref{prop-3} provides yet another proof of the fact that a Lie $p$-algebra
$L$ of restricted cohomological dimension one is infinite-dimensional (which 
follows also from Lemma \ref{lemma-fin}), without appealing to any computation of Hochschild cohomology. Indeed, suppose the 
contrary, and take in (\ref{eq-ss}) $n = \dim L + 1$. Then the left-hand side 
and the first direct summand at the right-hand side of the isomorphism vanish, 
and the second direct summand is isomorphic to $\Homol^1_*(L,M)$. Therefore, 
$\Homol^1_*(L,M) = 0$ for any restricted $L$-module $M$, i.e., $L$ is of 
restricted cohomological dimension zero, a contradiction.

Moreover, a stronger statement holds:

\begin{proposition}
A Lie $p$-algebra of restricted cohomological dimension one has infinite 
(ordinary) cohomological dimension.
\end{proposition}

\begin{proof}
Let $L$ be a Lie algebra of restricted cohomological dimension one. Taking in 
(\ref{eq-ss}) $M=K$, we get 
$$
\Homol^n(L,K) \simeq 
\Big(\bigwedge\nolimits^n L\Big)^\star \oplus 
\Hom\Big(\bigwedge\nolimits^{n-1} L,\> \Homol^1_*(L,K)\Big) .
$$
Either by Lemma \ref{lemma-fin}, or by the reasoning above, $L$ is 
infinite-dimensional, and thus $\bigwedge^n L$, and hence $\Homol^n(L,K)$, does 
not vanish for any $n \ge 1$.
\end{proof}

\section*{Acknowledgements}

Thanks are due to J\"org Feldvoss and Oleg Gatelyuk for useful discussions, and
to the anonymous referee for useful remarks correcting a few errors in the 
previous version of the manuscript. The hospitality of University of S\~ao Paulo
during an early stage of this work is gratefully acknowledged. This work was supported by the 
Regional Authority of the Moravian-Silesian Region (grant 01211/2016/RRC).


\begin{thebibliography}{MUZ}

\bibitem[BI]{benkart-isaacs} G. Benkart and I.M. Isaacs, 
\emph{On the existence of ad-nilpotent elements}, 
Proc. Amer. Math. Soc. \textbf{63} (1977), no.1, 39--40.

\bibitem[Be]{berkson} A.J. Berkson, 
\emph{The $u$-algebra of a restricted Lie algebra is Frobenius},
Proc. Amer. Math. Soc. \textbf{15} (1964), no.1, 14--15.

\bibitem[BK]{bokut-kukin} L.A. Bokut' and G.P. Kukin, 
\emph{Algorithmic and Combinatorial Algebra}, Kluwer, 1994.

\bibitem[Bou]{bourbaki} N. Bourbaki, 
\emph{Groupes et Alg\`ebres de Lie}, Chapitres 2 et 3, Hermann, Paris, 1972;
reprinted by Springer, 2006.


\bibitem[Ch1]{chwe-tams} B.-S. Chwe, 
\emph{Relative homological algebra and homological dimension of Lie algebras}, 
Trans. Amer. Math. Soc. \textbf{117} (1965), 477--493.

\bibitem[Ch2]{chwe} \bysame, \emph{On the commutativity of restricted Lie algebras}, Proc. Amer. Math. Soc. \textbf{16} (1965), no.3, 547.

\bibitem[Co]{lecturenotes} D.E. Cohen, 
\emph{Groups of cohomological dimension one}, 
Lect. Notes Math. \textbf{245} (1972).

\bibitem[CQ]{cuntz-quillen} J. Cuntz and D. Quillen,
\emph{Algebra extensions and nonsingularity}, 
J. Amer. Math. Soc. \textbf{8} (1995), no.2, 251--289.

\bibitem[D]{dzhu} A.S. Dzhumadil'daev, 
\emph{Cohomology of truncated coinduced representations of Lie algebras of positive characteristic}, 
Mat. Sbornik \textbf{180} (1989), no.4, 456--468 (in Russian); 
Math. USSR Sbornik \textbf{66} (1990), no.2, 461--473 (English translation).

\bibitem[EK]{elduque-kochetov} A. Elduque and M. Kochetov, 
\emph{Gradings on Simple Lie Algebras}, AMS, 2013.

\bibitem[Fa1]{farns-semilin} R. Farnsteiner, 
\emph{Restricted Lie algebras with semilinear p-mappings}, 
Proc. Amer. Math. Soc. \textbf{91} (1984), no.1, 41--45.

\bibitem[Fa2]{farns-coho} \bysame, 
\emph{Cohomology groups of reduced enveloping algebras},
Math. Z. \textbf{206} (1991), 103--117.

\bibitem[FS]{fs} \bysame{} and H. Strade, \emph{Shapiro's lemma and its consequences in the cohomology theory of modular Lie algebras}, Math. Z. \textbf{206} (1991), 153--168.

\bibitem[Fe]{feldman} G.L. Feldman, \emph{Ends of Lie algebras}, 
Uspekhi Mat. Nauk \textbf{38} (1983), no.1, 199--200 (in Russian); 
Russ. Math. Surv. \textbf{38} (1983), no.1, 182--184 (English translation).

\bibitem[FP]{fp} E.M. Friedlander and B.J. Parshall, 
\emph{Modular representation theory of Lie algebras}, 
Amer. J. Math. \textbf{110} (1988), no.6, 1055--1093.

\bibitem[G]{gein} A.G. Gein, 
\emph{The modular law and relative complements in the lattice of subalgebras of a Lie algebra}, 
Izv. VUZ Mat. 1987, no.3,	18--25 (in Russian); 
Soviet Math. (Izv. VUZ) \textbf{31} (1987), no.3, 22--32 (English translation).

\bibitem[Hoch]{hochsch} G. Hochschild, 
\emph{Representations of restricted Lie algebras of characteristic $p$},
Proc. Amer. Math. Soc. \textbf{5} (1954), no.4, 603--605.

\bibitem[Hol]{holm} T. Holm, \emph{Hochschild cohomology rings of algebras $k[X]/(f)$}, Beitr\"age Algebra Geom. \textbf{41} (2000), no.1, 291--301.

\bibitem[J]{jacobson} N. Jacobson, \emph{Lie Algebras}, Interscience Publ., 1962; reprinted by Dover, 1979.

\bibitem[K]{kaplansky} I. Kaplansky, \emph{Lie algebras of characteristic $p$},
Trans. Amer. Math. Soc. \textbf{89} (1958), no.1, 149--183.

\bibitem[L]{latyshev} V.N. Latyshev, 
\emph{On zero divisors and nilelements in Lie algebras},
Sibirsk. Mat. Zh. \textbf{4} (1963), no.4, 830--836 (in Russian).

\bibitem[MUZ]{muz} A.A. Mikhalev, U.U. Umirbaev, and A.A. Zolotykh, 
\emph{A Lie algebra with cohomological dimension one over a field of prime characteristic is not necessarily free}, First International Tainan-Moscow Algebra Workshop (ed. Y. Fong et al.), De Gruyter, 1996, 257--264.

\bibitem[M]{muzere} M. Muzere, \emph{Relative Lie algebra cohomology revisited},
Proc. Amer. Math. Soc. \textbf{108} (1990), no.3, 665--671.

\bibitem[P1]{premet} A.A. Premet, 
\emph{On Cartan subalgebras of Lie $p$-algebras}, 
Izv. Akad. Nauk SSSR Ser. Mat. \textbf{50} (1986), N4, 788--800 (in Russian); 
Math. USSR Izvestija \textbf{29} (1987), no.1, 145--157 (English translation).

\bibitem[P2]{premet-bssr} \bysame, 
\emph{Absolute zero divisors in Lie algebras over a perfect field},
Doklady Akad. Nauk BSSR \textbf{31} (1987), no.10, 869--872 (in Russian).

\bibitem[Sch]{schafer} R.D. Schafer,
\emph{An Introduction to Nonassociative Algebras},
Academic Press, 1966; reprinted in a slightly corrected form by Dover, 1995.

\bibitem[St]{strade} H. Strade, 
\emph{Einige Vereinfachungen in der Theorie der modularen Lie-Algebren},
Abh. Math. Sem. Univ. Hamburg \textbf{54} (1984), 257--265.

\bibitem[SF]{strade-farn} \bysame{} and R. Farnsteiner,
\emph{Modular Lie Algebras and Their Representations}, Marcel Dekker, 1988.

\bibitem[V]{varea87} V.R. Varea, 
\emph{Lie algebras none of whose Engel subalgebras are in intermediate position},
Comm. Algebra \textbf{15} (1987), no.12, 2529--2543.

\bibitem[Wa]{waterhouse} W.C. Waterhouse, 
\emph{Automorphism schemes and forms of Witt Lie algebras},
J. Algebra \textbf{17} (1971), no.1, 34--40.

\bibitem[We]{weibel} C. Weibel, \emph{An Introduction to Homological Algebra},
Cambridge Univ. Press, 1994.

\bibitem[Z1]{z1} 
P. Zusmanovich, \emph{A converse to the Second Whitehead Lemma},
J. Lie Theory \textbf{18} (2008), no.2, 295--299; 
Erratum: \textbf{24} (2014), no.4, 1207--1208; arXiv:0704.3864.

\bibitem[Z2]{z2} \bysame, \emph{A converse to the Whitehead Theorem},
J. Lie Theory \textbf{18} (2008), no.4, 811--815; arXiv:0808.0212. 

\end{thebibliography}
\end{document}